\documentclass[12pt,reqno]{amsart}
\usepackage{amsmath,amsthm,amssymb,amsfonts,amscd}
\usepackage{mathrsfs}
\usepackage{bbm}
\usepackage{bbding}
\usepackage[backref=page]{hyperref}
\usepackage{geometry}
\usepackage{color}
\usepackage{xcolor}

\usepackage{picture,epic}
\usepackage{tikz}

\usepackage{etoolbox}
\usepackage{orcidlink}

\newif\ifrevisionmode
\revisionmodetrue

\numberwithin{equation}{section}

\theoremstyle{plain}
\newtheorem{theorem}{Theorem}[section]
\newtheorem{lemma}[theorem]{Lemma}
\newtheorem{corollary}[theorem]{Corollary}
\newtheorem{proposition}[theorem]{Proposition}

\theoremstyle{definition}
\newtheorem{conjecture}[theorem]{Conjecture}

\theoremstyle{remark}
\newtheorem{remark}[theorem]{Remark}

\renewcommand{\Re}{\operatorname{Re}}
\renewcommand{\Im}{\operatorname{Im}}

\newcommand{\sym}{\operatorname{sym}}

\newcommand{\GL}{\operatorname{GL}}
\newcommand{\SL}{\operatorname{SL}}

\newcommand{\dd}{\mathrm{d}}

\makeatletter
\def\@tocline#1#2#3#4#5#6#7{\relax
  \ifnum #1>\c@tocdepth 
  \else
    \par \addpenalty\@secpenalty\addvspace{#2}%
    \begingroup \hyphenpenalty\@M
    \@ifempty{#4}{%
      \@tempdima\csname r@tocindent\number#1\endcsname\relax
    }{%
      \@tempdima#4\relax
    }%
    \parindent\z@ \leftskip#3\relax \advance\leftskip\@tempdima\relax
    \rightskip\@pnumwidth plus4em \parfillskip-\@pnumwidth
    #5\leavevmode\hskip-\@tempdima
      \ifcase #1
       \or\or \hskip 1em \or \hskip 2em \else \hskip 3em \fi%
      #6\nobreak\relax
    \hfill\hbox to\@pnumwidth{\@tocpagenum{#7}}\par
    \nobreak
    \endgroup
  \fi}
\makeatother

\begin{document}

\title{Quadratic Forms of Modular Forms}
\author{Shenghao Hua~\orcidlink{0000-0002-7210-2650}}
\address[1]{Shanghai Institute for Mathematics and Interdisciplinary Sciences (SIMIS), Shanghai, 200433, China}
\address[2]{Research Institute of Intelligent Complex Systems, Fudan University, Shanghai, 200433, China}
\email{huashenghao@vip.qq.com}


\begin{abstract}
In this paper, we study quadratic forms in spaces of holomorphic cusp forms.
We show, conditionally, that when two quadratic forms in Hecke eigenforms share no common diagonal terms, their inner product is expected to converge to the sum of the products of their common off-diagonal coefficients.
This phenomenon could be interpreted as a mixed $L^4$-norm problem.

We also define the $\ell^p$-norm of a holomorphic cusp form via its expansion with respect to an orthonormal Hecke basis.
We then establish a conditional upper bound for the $\ell^p$-norm, and deduce that the coefficients of quadratic forms of holomorphic cusp forms in the Hecke basis are not uniformly small, being dominated by small-amplitude components. This behavior is consistent with the expected distribution of orthogonal families of $L$-functions.

\end{abstract}



\maketitle

\section{Introduction} \label{sec:Intr}

Modular forms originated from the theory of elliptic functions in the 19th century and have since developed into a bridge connecting number theory, algebraic geometry, and representation theory. They play a central role in modern mathematics, profoundly driving the resolution of many major theories and conjectures.
In particular, modular forms form graded algebras, where the product yields elements of different gradings in the graded rings. In the corresponding graded linear spaces, expansions can be made.
A very natural question is how to describe the corresponding properties.
Let \(\mathbb{H}\) denote the upper half-plane and \(\Gamma = \mathrm{SL}_2(\mathbb{Z})\) the full modular group.

For \(k \geq 12\), the Petersson inner product on \(S_k\) is defined for \(h_1, h_2 \in S_k\) by
\[
\langle h_1, h_2 \rangle := \int_{\Gamma \backslash \mathbb{H}} y^k h_1(z) \overline{h_2(z)} \, d\mu(z),
\]
where the hyperbolic measure \(d\mu(z)\) is given by
\[
d\mu(z) = \frac{dx\, dy}{y^2}.
\]

The $L^4$-norm problem, as conjectured in~\cite[Conjecture 1.2]{BlomerKhanYoung2013}, where it is expected that the \( L^4 \)-norm of \( f \) is asymptotically 2.
Blomer, Khan, and Young~\cite{BlomerKhanYoung2013} proved the upper bound
\[
  \int_{\Gamma \backslash \mathbb{H}} y^{2k_0} |f(z)|^4 \, \dd\mu(z)
  = O(k_0^{1/3+\varepsilon}).
\]
Assuming the GRH, Zenz~\cite{Zenz2023} improved this to
\[
  \int_{\Gamma \backslash \mathbb{H}} y^{2k_0} |f(z)|^4 \, \dd\mu(z) = O(1).
\]
Motivated by this and the joint distribution of automorphic forms (see \cite{HuaHuangLi2024,Huang2024}), we make the following conjecture for the inner products of quadratic forms in Hecke eigenforms.
\begin{conjecture}\label{conj:innerprod}
For \( L^2\)-normalized Hecke eigenforms $f_1,f_2,f_3,f_4$ of even weights, i.e. $\langle f_i, f_i\rangle=1$, satisfying $k_1+k_2=k_3+k_4$, we have
\begin{equation}\label{eqn:conj}
\langle f_1 f_2,\, f_3 f_4 \rangle
=
\delta_{\{f_1,f_2\}=\{f_3,f_4\}}
\bigl(1+\,\delta_{f_1=f_2}\bigr)
+ o(1),
\end{equation}
as $k_1+k_2 \to \infty$.
\end{conjecture}

In particular, when all four forms coincide, this reduces to the classical $L^4$-norm problem.
For other cases, Huang~\cite{Huang2024} proved the conjecture when $f_1=f_2\neq f_3=f_4$, as well as when
$\{f_1,f_2\}=\{f_3,f_4\}\neq \{f_1\}$.
Here we complete the proof in all remaining cases in which the four forms are not all identical.

\begin{theorem}\label{thm:Heckejoint}
Let $\#\{f_1,f_2,f_3,f_4\}\ge 2$.
Assume GRH.
We make the following assumptions:
\begin{enumerate}
    \item[(i)] Assume the Generalized Ramanujan Conjecture when $\{f_1,f_2\} = \{f_3,f_4\} \neq \{f_1\}$.

    \item[(ii)] Assume $f_1 \times f_2 \nsim f_3 \times f_4$ when $\#\{f_1,f_2,f_3,f_4\} = 4$.

    \item[(iii)] Assume the analytic continuation of triple product $L$-functions involving symmetric squares in the following cases:
    \[
    \#\{f_1,f_2\} = 2 \quad\text{or}\quad \#\{f_3,f_4\} = 2,
    \]
    and $\{f_1,f_2\} \neq \{f_3,f_4\}$.
\end{enumerate}
Under these assumptions, Conjecture~\ref{conj:innerprod} holds.
\end{theorem}

An interesting phenomenon in \cite{Huang2024} is that the case of $f_1=f_2$ and $f_3=f_4$ is different to the case of $f_1=f_3$ and $f_2=f_4$. The reason is the product of modular forms is another modular form.

Let \(k_1, k_2 \geq 12\) be even integers.
For each \(i = 1, 2\), denote by \(S_{k_i}\) the space of holomorphic cusp forms of weight \(k_i\) on the modular surface \(\Gamma \backslash \mathbb{H}\).
For \(f \in S_{k_1}\) and \(g \in S_{k_2}\), we know that the product \(fg\) is a modular form of weight \(k_1 + k_2\). Moreover, due to the vanishing condition at the cusp, \(fg\) is itself a cusp form.
Since the first Fourier coefficient of any Hecke eigenform $h$ is non-zero,
and the first Fourier coefficient of the product vanishes, we know it can not be Hecke eigenform.
We have the decomposition
\[
fg = \sum_{h \in H_{k_1 + k_2}} \langle fg, h \rangle \, h,
\]
where \(H_{k_1 + k_2}\) is a \( L^2\)-normalized Hecke basis of \(S_{k_1 + k_2}\).
Then we have \(\sum_{h \in H_{k_1 + k_2}} \langle fg, h \rangle = 0 \).

When we expand modular forms in the Hecke basis, we obtain a bijection between the sequence of coefficients and the modular form. Consequently, we can define the $\ell^p$-norm as the sum of the $p$-th powers of the absolute values of the coordinate components.
In the case of the quadratic forms in modular cusp forms, we show a conditional upper bound for the $\ell^p$-norm, and deduce the coefficients in the Hecke
basis exhibit a lack of uniformly bounded distribution.

\begin{remark}
If one removes $\varepsilon$ via Harper's method~\cite{Harper2013}, a $\log\log$ power from \eqref{eqn:sumofcoefficients2a} persists. This is because the spectral parameters of the summation depend on those of the base forms, a dependence absent in the fixed-form case.
\end{remark}

Let \( N \in \mathbb{Z}^+ \).
Suppose we are given complex coefficients \( a_{i,j} \) for \( 1 \leq i, j \leq N \), satisfying the symmetry condition \( a_{i,j} = a_{j,i} \).
Let \( k_i \geq 12 \) be even positive integers such that there exists \( k \) with
\[
k_i + k_j = k \quad \text{whenever } a_{i,j} \neq 0.
\]
We then define the quadratic form
\[
Q(x_1, \dots, x_N) = \sum_{i,j=1}^{N} a_{i,j} x_i x_j.
\]

Let \( H_{k_i} \) be a $L^2$-normalized Hecke eigenbasis of the cusp form space \( S_{k_i} \).
For each \( 1 \leq i \leq N \), assuming \( f_i \in H_{k_i} \),
then \( Q(f_1, \dots, f_N) \) is a cusp form of weight \( k \).
For convenience, we set
$
\sum_{i,j}
\left|a_{i,j} \right|
\leq B
$
for some constant \( B > 0 \).
Using Minkowski's inequality $\|x+y\|_p\leq \|x\|_p+\|y\|_p$,
the \(\ell^p\)-norm estimate can be reduced to the case of a sum
of terms of the form \(f^2\) and \(fg\).
For such quadratic form, we have the following result.

\begin{theorem}\label{thm:lpnormformodular}
Assuming the GRH for certain \( L \)-functions, and the analytic continuation of triple product \( L \)-functions involving symmetric squares; precise statements will be given in Theorem~\ref{thm:hecke}.
  Let $\varepsilon>0$. Let \(k \to \infty\), for any \(p>0\), we have   \begin{multline}\label{eqn:upperforQ}
    \|Q(f_1, \dots, f_N)\|_{\ell^p,H_k}
    :=
    \left(\sum_{h\in H_{k}}
    |\langle Q(f_1, \dots, f_N), h \rangle|^{p}\right)^{1/p}
   \\ \ll_{N,B,\varepsilon}
   k^{\frac{1}{p}-\frac{1}{2}}
    \left(\log^{-\frac{2-p}{4}+(p+1)\varepsilon} k
    +\log^{-\frac{2-p}{8}+(p+1)\varepsilon} k \right).
  \end{multline}
Moreover, if we always have $a_{i,i}= 0$, then we have
\begin{equation}\label{eqn:upperforQpure}
    \|Q(f_1, \dots, f_N)\|_{\ell^p,H_k}
\ll_{N,B,\varepsilon}
   k^{\frac{1}{p}-\frac{1}{2}}
\log^{-\frac{2-p}{8}+(p+1)\varepsilon} k,
\end{equation}
  and for any $p>2$, we have
\[
  \|Q(f_1, \dots, f_N)\|_{\ell^p,H_k}=o(1).
\]
\end{theorem}

%

\begin{remark}
  The assumption of analytic continuation is unnecessary in the case of diagonal quadratic forms $Q$.
\end{remark}

In order to deduce a lower bound for the \(\ell^0\)-norm from an upper bound on the \(\ell^p\)-norm,
we set the coefficients to satisfy
$$a:=\sum_{i,j}\frac{a_{i,j}}{L(1,\sym^2 f_i)L(1,\sym^2 f_j)}\neq 0.$$
Then we have the following result.

\begin{theorem}\label{thm:nonsolution}
Assuming the GRH for certain \( L \)-functions, and the analytic continuation of triple product \( L \)-functions involving symmetric squares; precise statements will be given in Theorem~\ref{thm:hecke}.
For any \(L \in \mathbb{Z}^+ \), any $\varepsilon>0$, there exists a constant $ K=K_0L^{2+\varepsilon}$,
$K_0$ depends on $a$ and $N$, such that for all \( k > K \), there is no solution
\[
(c_1, \dots, c_{\dim S_k})
\]
with at most \( L \) nonzero coordinates satisfying
\[
\sum_{\phi_{k,r} \in H_k} c_r \phi_{k,r} = Q(f_1, \dots, f_N).
\]
\end{theorem}

\begin{remark}
  The assumption of analytic continuation is unnecessary in the case of diagonal quadratic forms $Q$.
\end{remark}

When Eisenstein series are included, Duke~\cite{Duke1999} and Ghate~\cite{Ghate2000} proved that the product of two Hecke eigenforms for the full modular group is itself a Hecke eigenform in only 16 cases. Beyerl, James, and Xue~\cite{BeyerlJamesXue2014} considered the Rankin--Cohen bracket, while Joshi and Zhang~\cite{JoshiZhang2019} investigated the case of Hilbert modular forms. Bao~\cite{Bao2019} extended the result to certain binary quadratic forms in holomorphic cusp forms.

The following form comes from Theorem~\ref{thm:Heckejoint}.
\begin{corollary}\label{cor:asy}
For two quadratic forms
\[
Q(x_1,\dots,x_N) = \sum_{i,j=1}^{N} a_{i,j} x_i x_j,\qquad
Q'(x_1,\dots,x_N) = \sum_{i,j=1}^{N} a_{i,j}' x_i x_j,
\]
assume that they have no common diagonal terms, i.e., $a_{i,i}
a_{i,i}' = 0$ for all $i$.
Let $k$ large enough, and $(f_1,\dots,f_N) \in \prod_{i=1}^{N}H_{k_i}$ with $k_i+k_j=k$ if some $a_{i,j}\neq 0$ or $a_{i,j}'\neq 0$.
For those $(f_i,f_j,f_k,f_l)$ with $a_{i,j}a_{k,l}' \neq 0$, we impose the same conditions as in
Theorem~\ref{thm:Heckejoint}, and we have
\[
\langle Q(x_1,\dots,x_N),\,Q' (f_1,\dots,f_N)\rangle = \sum_{i \neq j} a_{i,j}a_{i,j}' + o(1).
\]
\end{corollary}

Here we show how Theorem~\ref{thm:lpnormformodular} implies Theorem~\ref{thm:nonsolution}.

\begin{proof}[Proof of Theorem~\ref{thm:nonsolution} assuming Theorem~\ref{thm:lpnormformodular}]
Recall that the first Fourier coefficient of every $L^2$-normalized Hecke eigenform $\phi_r$ is $\frac{1}{L(1,\sym^2 \phi_r)}$,
so the second Fourier coefficient of \( Q(f_1, \dots, f_N) \) equals $a$.

If a solution exists, then the second Fourier coefficient of \( Q(f_1, \dots, f_N) \) must also equal \( \sum_{r=1}^{\dim S_k} \frac{c_r \lambda_{\phi_r}(2)}{L(1,\sym^2 \phi_r)} \), where \( \lambda_{\phi_r}(2) \) denotes the second Fourier coefficient of \( \phi_r \).
Recall Deligne's bound, hence there exists some \( r_0 \) such that
$$|c_{r_0}| \geq \frac{|a|L(1,\sym^2 \phi_{r_0})}{2L},$$
and consequently, for any $\varepsilon>0$ and any $p>0$,
\[
\|Q(f_1, \dots, f_N)\|_{\ell^p,H_k} \gg \frac{|a|}{L\log^{\varepsilon} k}.
\]
This contradicts Theorem~\ref{thm:lpnormformodular}.
\end{proof}

We reduce the proof of Theorem~\ref{thm:lpnormformodular} to the case of Hecke eigenforms, establishing the decay of the $\ell^p$-norm through the study of mixed moments of $L$-functions.

\begin{theorem}\label{thm:hecke}
Let \(p>0\).
Let \(k_1, k_2 \geq 12\) be even integers.
Let \( f \in H_{k_1} \) and \( g \in H_{k_2} \), where \(H_{k}\) is a \( L^2\)-normalized Hecke basis of \(S_{k}\).
Assuming the analytic continuation of $L(s, \sym^2 f \times \sym^2 g\times \sym^2 h)$ and the GRH for
\[
\begin{aligned}
&L(s, h), \quad
L(s, f \times g \times h), \quad
L(s, \sym^2 f), \\
&L(s, \sym^2 g), \quad
L(s, \sym^2 h), \quad
L(s, \sym^2 f \times \sym^2 g), \\
&L(s, \sym^2 f \times \sym^2 h), \quad
L(s, \sym^2 g\times \sym^2 h), \quad
L(s, \sym^2 f \times \sym^2 g\times \sym^2 h),
\end{aligned}
\]
for all \(h \in H_{k_1 + k_2}\), any $\varepsilon>0$,
as \(\max\{k_1, k_2\} \to \infty\), we have
\begin{equation}\label{eqn:tripleinner}
  \left(\sum_{h \in H_{k_1 + k_2}} |\langle f g, h \rangle|^{p}\right)^{1/p} \ll_{\varepsilon}
  (k_1+k_2)^{\frac{1}{p}-\frac{1}{2}}
  \log^{-\frac{2-p}{8}+(p+1)\varepsilon}(k_1 + k_2).
\end{equation}

Assuming the GRH for \[
L(s, h), \quad
L(\sym^2 f \times h), \quad
L(\sym^4 f \times h), \quad
L(s, \sym^2 f), \quad
L(s, \sym^4 f),
\]
for all \(h \in H_{2 k_1}\), any $\varepsilon>0$,
as \(k_1 \to \infty\), we have
\begin{equation}\label{eqn:mixinner}
 \left(\sum_{h \in H_{2 k_1}} |\langle f^2, h \rangle|^{p}\right)^{1/p} \ll_{\varepsilon}
k_1^{\frac{1}{p}-\frac{1}{2}}
 \log^{-\frac{2-p}{4}+(p+1)\varepsilon} k_1.
\end{equation}
\end{theorem}

\begin{proof}[Proof of Theorem~\ref{thm:lpnormformodular} assuming Theorem~\ref{thm:hecke}]
By Minkowski's inequality, we have
\begin{equation*}
\|Q(f_1, \dots, f_N)\|_{\ell^p,H_k}
\\ \leq \sum_{i,j}|a_{i,j}|
\left( \sum_{h \in H_{k}} |\langle f_i f_j,h \rangle|^{p} \right)^{1/p}.
\end{equation*}
Then, applying Theorem~\ref{thm:hecke} yields the desired estimate \eqref{eqn:upperforQ}.
In particular, all $a_{i,i}=0$ implies that the contribution from \( |\langle f^2, h \rangle| \) vanishes, hence we obtain a better bound arising solely from \eqref{eqn:mixinner}, without any contribution from \eqref{eqn:tripleinner},
 which leads to \eqref{eqn:upperforQpure}.
\end{proof}

We prove Theorem~\ref{thm:hecke} by studying real moments of the following $L$-functions.
We prove this latter result in \S\ref{sec:moments} via Soundararajan's method~\cite{Soundararajan2009}.
A similar argument, carried out in \S\ref{sec:quadjoint}, yields Theorem~\ref{thm:Heckejoint}.

\begin{proposition}\label{prop:LboundSound}
  Under the assumptions of Theorem~\ref{thm:hecke}, including the analytic continuation and the GRH for the relevant \(L\)-functions.
  For $l,l_1,l_2> 0$, and any $\varepsilon>0$, we have that
  \begin{equation}\label{eqn:uppersymLSound}
   \frac{1}{k_1}
  \sum_{h \in H_{2k_1}}
  L\left(\tfrac{1}{2}, h\right)^{l_1} L\left(\tfrac{1}{2}, \sym^2 f \times h\right)^{l_2}
  \ll_{\varepsilon} \sqrt{l_1^2+l_2^2}\log ^{\frac{l_1(l_1 - 1)}{2} + \frac{l_2(l_2 - 1)}{2} + (l_1^2+l_2^2+1)\varepsilon}k_1,
  \end{equation}
  and
  \begin{equation}\label{eqn:uppertripleLSound}
   \frac{1}{k_1+k_2}
  \sum_{h \in H_{k_1+k_2}} L\left(\tfrac{1}{2}, f \times g \times h\right)^l
   \ll_{\varepsilon} l\log (k_1+k_2)^{\frac{l(l-1)}{2} + (l^2+1)\varepsilon}.
  \end{equation}
\end{proposition}

\begin{proof}[Proof of Theorem~\ref{thm:hecke}]
Watson's formula~\cite{Watson2008} gives
\begin{equation*}
  |\langle f^2 , h\rangle|^2
  \ll \frac{1}{k_1} \frac{L(1/2, h) L(1/2, \sym^2 f \times h)}
    {L(1, \sym^2 f)^2 L(1, \sym^2 h)},
\end{equation*}
and
\begin{equation*}
  |\langle f g , h\rangle|^2
  \ll \frac{1}{k_1 + k_2} \frac{L(1/2, f \times g \times h)}
    {L(1, \sym^2 f) L(1, \sym^2 g) L(1, \sym^2 h)},
\end{equation*}
where the non-negativity of the central \(L\)-values follows from Lapid's theorem~\cite{Lapid2003}.
Under the GRH, for $\phi\in H_k$ we have $(\log\log k)^{-1} \ll  L(1, \sym^2 \phi) \ll (\log\log k)^{3}$ (see \cite[Theorem 3]{LauWu2006}).
Then Theorem~\ref{thm:hecke} follows from Proposition~\ref{prop:LboundSound}.
\end{proof}

From Theorems~\ref{thm:Heckejoint} and \ref{thm:lpnormformodular}, we could see the coefficients of quadratic forms of holomorphic cusp forms in the Hecke basis are not uniformly small, being dominated by small-amplitude components.
Let \(A\) denote the upper bound for the \(L^4\)-norm established by Zenz~\cite{Zenz2023}. By the Cauchy--Schwarz inequality, we have $A\ge 1$.

\begin{theorem}\label{thm:large}
Let
\[
\sum_{i\ne j} a_{i,j}
\notin
\left\{
\sum_i a_{i,i} y_i : y_i\in[1,A]
\right\}.
\]
where $\mathcal{A}=[1,A]$.
 Assuming the above conditions, there exist coordinates for the quadratic form \(Q\) with respect to the Hecke orthogonal basis that are of larger-than-average order as \(k\to\infty\).
\end{theorem}

\begin{proof}[Proof of Theorem~\ref{thm:large}]
  In this case, from \cite{Zenz2023} and Theorem~\ref{thm:Heckejoint}, we know the leading term for the \(\ell^2\)-norm, which corresponds to the mixed \(L^4\)-norm associated with \(Q\), is nonzero and of constant order.
From Theorem~\ref{thm:lpnormformodular}, we see that the average order of coordinates of \(Q\) is smaller than \(1/\sqrt{k}\), up to a logarithmic factor.
Hence there exists at least one coordinate whose magnitude is significantly larger than the average order.
\end{proof}

\section{Upper Bounds for Moments of
\texorpdfstring{$L$}{l}-Functions}\label{sec:moments}

In this section, we establish Proposition~\ref{prop:LboundSound} by applying Soundararajan's method~\cite{Soundararajan2009}.
For related results and alternative approaches, Lester and Radziwi{\l\l}~\cite{LesterRadziwill2020} studied quantum unique ergodicity for half-integral weight automorphic forms; Huang and Lester~\cite{HuangLester2023} investigated the quantum variance of dihedral Maass forms; Blomer, Brumley, and Khayutin~\cite{BlomerBrumleyKhayutin2022} proved the joint equidistribution conjecture proposed by Michel and Venkatesh in their 2006 ICM proceedings article~\cite{MichelVenkatesh2006}; and Hua, Huang, and Li~\cite{HuaHuangLi2024} established a case of their joint Gaussian moment conjecture, and the holomorphic version is discussed by Huang~\cite{Huang2024}.
Chatzakos, Cherubini, Lester, and Risager~\cite{CCLS2025} obtained a logarithmic improvement on Selberg's longstanding bound for the error term in the hyperbolic circle problem counting function over Heegner points with varying discriminants.

In this chapter, we use \( p \) to denote a prime number, as opposed to its meaning in Theorem~\ref{thm:hecke}.
Let $\lambda_h(n)$ be the $n$-th Hecke eigenvalue of $h$.
We will use the following lemma, which is a consequence of Petersson's formula.

\begin{lemma}[{\cite[Lemma 2.1]{RudnickSoundararajan2006}}]\label{lemma:petersson}
  Let \( k \) be a large even integer. For natural numbers \( m \) and \( n \) satisfying \( mn \leq k^2 / 10^4 \), we have
  \[
    \frac{2\pi^2}{k - 1} \sum_{h \in H_k} \frac{\lambda_h(m)\lambda_h(n)}{L(1, \sym^2 h)} = \delta_{m=n} + O(e^{-k}).
  \]
\end{lemma}

Let $\alpha_f,\beta_f$, $\alpha_g,\beta_g$, and $\alpha_h,\beta_h$ denote the Satake parameters for $f$, $g$, and $h$, respectively.

\begin{lemma}\label{lemma:sumusesummationformula}
Assume the GRH for $L(s, \sym^2 h)$.
Let \( r \in \mathbb{N} \). Then, for \( x \leq (k_1 + k_2)^{\frac{1}{10r}} \) and any real numbers \( a_p \ll p^{\varepsilon} \) for any \( \varepsilon > 0 \), we have
\begin{equation}
\sum_{h \in H_{k_1 + k_2}}
\left( \sum_{p \leq x} \frac{a_p \lambda_h(p)}{p^{1/2}} \right)^{2r}
\ll \frac{(2r)!}{r! \, 2^r} (k_1 + k_2)(\log\log(k_1 + k_2))^3
\left( \sum_{p \leq x} \frac{a_p^2}{p} \right)^r.
\end{equation}
\end{lemma}

\begin{proof}
Under the GRH, we have \( L(1, \sym^2 h) \ll (\log\log(k_1 + k_2))^3 \).
Using the identity
\[
\lambda_h(p^l) = \sum_{0 \leq m \leq l} \alpha_h(p)^m \beta_h(p)^{l - m},
\]
we obtain
\[
\lambda_h(p)^k = (\alpha_h(p) + \beta_h(p))^k
= \sum_{\substack{0 \leq l \leq k \\ l \equiv k \, (\mathrm{mod} \, 2)}} D_{k,l} \lambda_h(p^l),
\]
where
\[
D_{k,l} = \frac{k!}{\left( \frac{k+l}{2} \right)! \left( \frac{k-l}{2} \right)!}
- \sum_{0 < m \leq \frac{k-l}{2}} D_{k, l+2m}, \quad \text{with} \quad D_{k,k} = 1.
\]
So,
\[
D_{k,l} = \frac{k! (l+1)}{\left( \frac{k+l}{2} + 1 \right)! \left( \frac{k - l}{2} \right)!}.
\]

Let \( a_n = \prod_{p^j \| n} a_p^j \). Then we have
\begin{multline}\label{eqn:counting}
\sum_{h \in H_{k_1 + k_2}}
\left( \sum_{p \leq x} \frac{a_p \lambda_h(p)}{p^{1/2}} \right)^{2r}
\\=
\sum_{\substack{n = p_1^{e_1} \cdots p_q^{e_q} \\ p_i \leq x \\ \sum e_i = 2r}}
\frac{a_n}{n^{1/2}}
\sum_{\substack{0 \leq l_i \leq e_i \\ l_i \equiv e_i \, (\mathrm{mod} \, 2)}}
\frac{(2r)! \prod_{i=1}^q (l_i + 1)}
{\prod_{i=1}^q \left( \left( \frac{e_i + l_i}{2} + 1 \right)! \left( \frac{e_i - l_i}{2} \right)! \right)}
\sum_{j} \frac{\lambda_h(p_1^{l_1} \cdots p_q^{l_q})}{L(1, \sym^2 h)}.
\end{multline}

Using Lemma~\ref{lemma:petersson}, this is equal to
\[
\frac{k_1 + k_2 - 1}{2\pi^2}
\sum_{\substack{n = p_1^{2f_1} \cdots p_q^{2f_q} \\ p_i \leq x \\ \sum f_i = r}}
\frac{(2r)!}{\prod_{i=1}^q \left( f_i! (f_i + 1)! \right)}
\frac{a_n}{n^{1/2}} + O(e^{-0.99(k_1+k_2)}).
\]

Since \( a_{p_1^{2f_1} \cdots p_q^{2f_q}} \geq 0 \), and using the inequality \( (n+1)! \geq 2^n \), we have
\[
\frac{(2r)!}{\prod_{i=1}^q f_i! (f_i + 1)!}
\leq \frac{(2r)!}{r!} \cdot \frac{r!}{\prod_{i=1}^q f_i! \cdot 2^{f_i}}
= \frac{(2r)!}{r! 2^r} \cdot \frac{r!}{\prod_{i=1}^q f_i!}.
\]

From the trivial bound
\[
\frac{e_i!}{\lceil \frac{e_i}{2} \rceil ! \, \lfloor \frac{e_i}{2} \rfloor !} \leq 2^{e_i},
\]
we finally obtain:
\begin{multline*}
\sum_{h \in H_{k_1 + k_2}} \frac{1}{L(1,\sym^2 h)}
\left( \sum_{p \leq x} \frac{a_p \lambda_h(p)}{p^{1/2}} \right)^{2r}
\ll \frac{(2r)!}{r! 2^r} (k_1 + k_2)
\sum_{\substack{n = p_1^{2f_1} \cdots p_q^{2f_q} \\ \sum f_i = r}}
\frac{r!}{\prod_{i=1}^q f_i!} \cdot \frac{|a_n|}{n^{1/2}} \\
\ll \frac{(2r)!}{r! 2^r}
(k_1 + k_2)
\left( \sum_{p \leq x} \frac{a_p^2}{p} \right)^r.
\end{multline*}
This completes the proof.
\end{proof}

Let
\begin{equation}\label{def:Lambdafgh}
  \Lambda_{f \times g \times h}(p^n)
  = \bigl(\alpha_f(p)^n + \beta_f(p)^n\bigr)
  \bigl(\alpha_g(p)^n + \beta_g(p)^n\bigr)
  \bigl(\alpha_h(p)^n + \beta_h(p)^n\bigr).
\end{equation}
In particular, we have the Hecke relation
\begin{equation}\label{eqn:heckerelation}
  \Lambda_{f \times g \times h}(p^2)
  = \Bigl( \Lambda_{\sym^2 f}(p) - 1 \Bigr)
  \Bigl( \Lambda_{\sym^2 g}(p) - 1 \Bigr)
  \Bigl( \Lambda_{\sym^2 h}(p) - 1 \Bigr).
\end{equation}

\begin{lemma}[{\cite[Theorem 2.1]{Chandee2009}}]\label{lemma:LogLfunctions}
  Under the assumptions of Theorem~\ref{thm:hecke}, including the GRH for
  $L(s, f \times g \times h)$, we have for $x > 10$:
  \begin{equation}\label{eqn:logLfunctions}
    \log L\biggl(\frac{1}{2}, f \times g \times h\biggr) \leq
    \sum_{p^n \leq x} \frac{\Lambda_{f \times g \times h}(p^n)}
    {n p^{n \left(\frac{1}{2} + \frac{1}{\log x}\right)}}
    \frac{\log \frac{x}{p^n}}{\log x}
    + O\biggl(\frac{\log (k_1 + k_2)}{\log x} + 1\biggr),
  \end{equation}
  where the implied constant is absolute.
\end{lemma}

\begin{lemma}\label{lemma:sumofcoefficients}
 Under the assumptions of Theorem~\ref{thm:hecke}, including the analytic continuation of $L(s, \sym^2 f \times \sym^2 g\times \sym^2 h)$ and the GRH for the relevant \(L\)-functions, the following estimates hold for $x \geq 2$:
  \begin{align}
    \sum_{p \leq x} \frac{\lambda_{\sym^2 f}(p)\lambda_{\sym^2 g}(p) \lambda_{\sym^2 h}(p)}{p}
    &= O(\log\log\log (k_1 + k_2)), \label{eqn:sumofcoefficients3} \\
        \sum_{p \leq x} \frac{\lambda_{\sym^2 f}(p) \lambda_{\sym^2 g}(p)}{p}
    &= O(\log\log\log (k_1 + k_2)), \label{eqn:sumofcoefficients2a} \\
    \sum_{p \leq x} \frac{\lambda_{\sym^2 f}(p) \lambda_{\sym^2 h}(p)}{p}
    &= O(\log\log\log (k_1 + k_2)), \label{eqn:sumofcoefficients2b} \\
    \sum_{p \leq x} \frac{\lambda_{\sym^2 g}(p) \lambda_{\sym^2 h}(p)}{p}
    &= O(\log\log\log (k_1 + k_2)), \label{eqn:sumofcoefficients2c} \\
    \sum_{p \leq x} \frac{\lambda_{\sym^2 f}(p)}{p}
    &= O(\log\log\log (k_1 + k_2)), \label{eqn:sumofcoefficients1a} \\
    \sum_{p \leq x} \frac{\lambda_{\sym^2 g}(p)}{p}
    &= O(\log\log\log k_1), \label{eqn:sumofcoefficients1b} \\
    \sum_{p \leq x} \frac{\lambda_{\sym^2 h}(p)}{p}
    &= O(\log\log\log k_2). \label{eqn:sumofcoefficients1c}
  \end{align}
\end{lemma}

\begin{proof}
  We establish the first bound \eqref{eqn:sumofcoefficients3} in detail; the others follow similarly using facts such as $\sym^2 f \ncong \sym^2 g \ncong \sym^2 h$.
  From \cite{GelbartJacquet1978}, we know that $\sym^2 f,~\sym^2 g,~\sym^2 g$ are self-dual cusp forms over $\SL_3(\mathbb{Z})$, and \cite{Ramakrishnan2014} establishes that $\sym^2 f \ncong \sym^2 g \ncong \sym^2 h$.

  Assuming the GRH for $L(s, \sym^2 f \times \sym^2 g\times \sym^2 h)$, the function $\log L(s, \sym^2 f \times \sym^2 g\times \sym^2 h)$ is analytic for $\Re(s) \geq \frac{1}{2} + \frac{1}{\log x}$. By a classical argument of Littlewood \cite[(14.2.2)]{Titchmarsh1986}, in this region we have
  \begin{equation}\label{eqn:logLbound}
    \left|\log L\left(s, \sym^2 f \times \sym^2 g\times \sym^2 h\right)\right|
    \ll \left(\Re(s) - \frac{1}{2}\right)^{-1} \log \left(k_1 + k_2 + |\Im(s)|\right).
  \end{equation}

  For $\Re(s) > 0$, we have
  \begin{equation*}
    \sum_{n} \frac{|\Lambda_{\sym^2 f}(n) \Lambda_{\sym^2 g}(n)\Lambda_{\sym^2 h}(n)|}{n^{1+s}} \ll 1,
  \end{equation*}
  and Deligne's bound yields
  \begin{equation*}
    \sum_{a \geq 2} \sum_{p^a \leq x} \frac{|\Lambda_{\sym^2 f}(p^a)
    \Lambda_{\sym^2 g}(p^a) \Lambda_{\sym^2 h}(p^a)|}{p^a} \ll 1.
  \end{equation*}

  Applying Perron's formula for $x \geq 2$ gives
  \begin{multline}
    \sum_{p \leq x} \frac{\lambda_{\sym^2 f}(p)\lambda_{\sym^2 g}(p) \lambda_{\sym^2 h}(p)}{p}
    = \sum_{p \leq x} \frac{\Lambda_{\sym^2 f}(p)\Lambda_{\sym^2 g}(p) \Lambda_{\sym^2 h}(p)}{p} \\
    = \frac{1}{2\pi i} \int_{1 - ix\log(k_1 + k_2 + x)}^{1 + ix\log(k_1 + k_2 + x)}
    \log L(s + 1, \sym^2 f \times \sym^2 g\times \sym^2 h) x^s \frac{\dd s}{s} \\
    + O\left(\frac{x \log x}{x \log(k_1 + k_2 + x)}\right)
    + O\left(\frac{x \sum_{p \text{ prime}} \frac{|\lambda_{\sym^2 f}(p)\lambda_{\sym^2 g}(p) \lambda_{\sym^2 h}(p)|}{p^2}}{x \log(k_1 + k_2 + x)}\right)
    + O(1).
  \end{multline}

  Shifting the contour to $\Re(s) = -\frac{1}{2} + \frac{1}{\log x}$, we encounter a simple pole at $s = 0$ with residue $\log L(1, \sym^2 f \times \sym^2 g\times \sym^2 h)$. The upper horizontal contour is bounded by
  \begin{multline}
    \ll \frac{1}{x \log(k_1 + k_2 + x)}
    \int_{-\frac{1}{2} + \frac{1}{\log x} + ix \log(k_1 + k_2 + x)}^{1 + ix \log(k_1 + k_2 + x)}
    |\log L(s + 1, \sym^2 f \times \sym^2 g\times \sym^2 h)| |x^s| |\dd s| \\
    \ll \frac{\log x \log(k_1 + k_2 + x \log(k_1 + k_2 + x))}{x \log(k_1 + k_2 + x)}
    \int_{-\frac{1}{2}}^{1} x^u \dd u \ll 1,
  \end{multline}
  and similarly for the lower horizontal contour.

  From \eqref{eqn:logLbound}, we obtain for $x \geq 2$:
  \begin{multline}
    \sum_{p \leq x} \frac{\lambda_{\sym^2 f}(p)\lambda_{\sym^2 g}(p)\lambda_{\sym^2 h}(p)}{p}
    = \log L(1, \sym^2 f \times \sym^2 g\times \sym^2 h) \\
    + O\left(1 + \frac{\log x}{\sqrt{x}} \int_{-x \log(k_1 + k_2 + x)}^{x \log(k_1 + k_2 + x)}
    \frac{\log (k_1 + k_2 + u)}{1 + |u|} \dd u \right).
  \end{multline}

  Applying this estimate twice yields for $z \geq (\log(k_1 + k_2))^3$:
  \begin{equation}
    \left| \sum_{(\log (k_1 + k_2))^3 < p \leq z} \frac{\lambda_{\sym^2 f}(p)\lambda_{\sym^2 g}(p)\lambda_{\sym^2 h}(p)}{p} \right| \ll 1.
  \end{equation}
  For $y \leq (\log (k_1 + k_2))^3$, we have
  \begin{equation}
    \left| \sum_{p \leq y} \frac{\lambda_{\sym^2 f}(p)\lambda_{\sym^2 g}(p)\lambda_{\sym^2 h}(p)}{p} \right| \ll \log\log\log (k_1 + k_2).
  \end{equation}
  This completes the proof of \eqref{eqn:sumofcoefficients3}.
\end{proof}

Using Lemma~\ref{lemma:sumofcoefficients}, for $2 \leq y \leq x$, $l > 0$, and distinct Hecke--Maass forms $f, g$, we have
\begin{equation}\label{eqn:sumofsumoverp}
  \sum_{y < p \leq x} \frac{l^2 \lambda_{\sym^2 f}(p)^2
  \lambda_{\sym^2 g}(p)^2}{p}
  = l^2 \log \frac{\log x}{\log y} + O(l\log\log\log (k_1 + k_2)).
\end{equation}

Before stating our next lemma, we introduce the following notation. For parameters $2 \leq y \leq x$, define
\begin{equation}
  \mathcal{P}(h;x,y) = \sum_{p \leq y} \frac{l \lambda_f(p)\lambda_g(p) \lambda_h(p)}{p^{\frac{1}{2} + \frac{1}{\log x}}} \left(1 - \frac{\log p}{\log x}\right),
\end{equation}
and let $\mathcal{A}(V;x) = \#\{h \in H_{k_1 + k_2} : \mathcal{P}(h;x,x) > V\}$. We also define the variance
\begin{equation}
  \sigma(k_1 + k_2)^2 = l^2 \log\log (k_1 + k_2).
\end{equation}

\begin{lemma}\label{lemma:AXY}
  Under the assumptions of Theorem~\ref{thm:hecke}, including the automorphy of $\sym^2(f\otimes g)$ and the GRH for the relevant \(L\)-functions.
  Let $C \geq 1$ be fixed and $\varepsilon > 0$ be sufficiently small. With the above notation, for all
  \[
  l\sqrt{\log\log (k_1 + k_2)} \leq V \leq C l\frac{\log (k_1 + k_2)}{\log\log (k_1 + k_2)},
  \]
  we have the bound
  \begin{equation}
    \mathcal{A}\left(V; (k_1 + k_2)^{\frac{1}{\varepsilon V}}\right)
    \ll (k_1 + k_2) \left(e^{-\frac{(1 - 2\varepsilon)V^2}{2\sigma(k_1 + k_2)^2}} (\log\log (k_1 + k_2))^3 + e^{-\frac{\varepsilon}{11}V\log V}\right).
  \end{equation}
\end{lemma}

\begin{proof}
  Throughout the proof, we assume $\varepsilon > 0$ sufficiently small, and consider the range
  \[
  l\sqrt{\log\log (k_1 + k_2)} \leq V \leq Cl \frac{\log (k_1 + k_2)}{\log\log (k_1 + k_2)}.
  \]

  Following Soundararajan's optimization method, we choose the length of our Dirichlet polynomial as $x = (k_1 + k_2)^{\frac{1}{\varepsilon V}}$. We decompose $\mathcal{P}(h;x,x) = \mathcal{P}_1(h) + \mathcal{P}_2(h)$, where $\mathcal{P}_1(h) = \mathcal{P}(h;x,z)$ with $z = x^{\frac{1}{\log\log (k_1 + k_2)}}$. This choice ensures $\sum_{z \leq p \leq x} \frac{1}{p} \ll \log\log\log(k_1 + k_2)$.

  Let $V_1 = (1 - \varepsilon)V$ and $V_2 = \varepsilon V$. If $\mathcal{P}(h;x,x) > V$, then either
  \begin{equation}\label{eqn:case1}
    \mathcal{P}_1(h) > V_1,
  \end{equation}
  or
  \begin{equation}\label{eqn:case2}
    \mathcal{P}_2(h) > V_2.
  \end{equation}

  Using Lemma~\ref{lemma:sumusesummationformula} and \eqref{eqn:sumofsumoverp}, we find that for parameters satisfying $r \leq \frac{\varepsilon V}{10} \log\log (k_1 + k_2)$ and $z \ll (k_1 + k_2)^{\frac{1}{10r}}$, the number of $h \in H_{k_1 + k_2}$ satisfying \eqref{eqn:case1} is bounded by
  \begin{equation}
    \frac{1}{V_1^{2r}} \sum_{h \in H_{k_1 + k_2}} \mathcal{P}_1(h)^{2r}
    \ll \frac{(2r)!}{V_1^{2r} r! 2^r} (k_1 + k_2) (\log\log (k_1 + k_2))^3 \sigma(k_1 + k_2)^{2r}.
  \end{equation}

  We consider two cases for the parameter $r$:
  \begin{itemize}
    \item For $V \leq \frac{\varepsilon}{10} \sigma(k_1 + k_2)^2 \log\log (k_1 + k_2)$, we take $r = \lfloor \frac{V_1^2}{2\sigma(k_1 + k_2)^2} \rfloor$.
    \item For larger $V$, we set $r = \lfloor \frac{\varepsilon V}{10} \rfloor$.
  \end{itemize}
  This yields the estimate
  \begin{equation*}
    \# \{ h \in H_{k_1 + k_2} : \mathcal{P}_1(h) > V_1 \}
    \ll (k_1 + k_2) \left(e^{-(1 - 2\varepsilon)\frac{V^2}{2\sigma(k_1 + k_2)^2}} (\log\log (k_1 + k_2))^3 + e^{-\frac{\varepsilon}{11} V\log V}\right).
  \end{equation*}

  To bound the number of $h$ satisfying \eqref{eqn:case2}, we take $r = \lfloor \frac{\varepsilon V}{10} \rfloor$, noting that $x \ll (k_1 + k_2)^{\frac{1}{10r}}$. Applying Lemma~\ref{lemma:sumusesummationformula} and \eqref{eqn:sumofsumoverp} again gives
  \begin{multline}
    \frac{1}{V_2^{2r}} \sum_{h \in H_{k_1 + k_2}} \mathcal{P}_2(h)^{2r}
    \ll (k_1 + k_2) (\log\log (k_1 + k_2))^3 \frac{(2r)!}{r!} \\
    \times \left(\frac{C}{V_2^2} \log\log\log (k_1 + k_2)\right)^r
    \ll (k_1 + k_2) e^{-\frac{\varepsilon}{11}V\log V}.
  \end{multline}
  Combining these estimates completes the proof.
\end{proof}

\subsection{Proof of Proposition~\ref{prop:LboundSound}}

\begin{proof}
Note that \eqref{eqn:uppersymLSound} is a special case of \cite[Proposition 5.1]{Huang2024}, obtained by setting the exponent of one of the $\mathrm{GL}(3) \times \mathrm{GL}(2)$ $L$-functions to zero.
The independence on $l_1$ and $l_2$ can be obtained similarly from the following proof.
It remains to prove \eqref{eqn:uppertripleLSound}.

Using the relation \eqref{eqn:heckerelation} and bounding the contribution from terms with $n \geq 3$, we obtain the decomposition
\begin{multline}\label{eqn:Lambdasumsym2fuj}
    \sum_{p^n \leq x} \frac{\Lambda_{f \times g \times h}(p^n)}{n p^{n(\frac{1}{2}+\frac{1}{\log x})}} \frac{\log \frac{x}{p^n}}{\log x}
    = \sum_{p \leq x} \frac{\lambda_f(p) \lambda_g(p) \lambda_h(p)}{p^{\frac{1}{2}+\frac{1}{\log x}}} \frac{\log \frac{x}{p}}{\log x} \\
    + \frac{1}{2} \sum_{p \leq \sqrt{x}} \frac{(\lambda_{\sym^2 f}(p)-1)(\lambda_{\sym^2 g}(p)-1)(\lambda_{\sym^2 h}(p)-1)}{p^{1+\frac{2}{\log x}}} \frac{\log \frac{x}{p^2}}{\log x}
    + O(1).
\end{multline}

Applying Lemma \ref{lemma:sumofcoefficients} to the second sum in \eqref{eqn:Lambdasumsym2fuj} yields
\begin{equation}\label{eqn:secondterms2}
    -\frac{1}{2}\log\log x + O(\log \log \log (k_1+k_2)).
\end{equation}

Let us define the following key quantities:
\begin{equation}
    \mu(k_1+k_2) = \left(-\frac{1}{2}+\varepsilon\right) l \log\log (k_1+k_2),
\end{equation}
and the $L$-function moment
\begin{equation}
    \mathcal{L}(h) = L(1/2, f \times g \times h)^l,
\end{equation}
with the counting function
\begin{equation}
    \mathcal{B}(V) = \#\{h \in H_{k_1+k_2} : \log \mathcal{L}(h) >V\}.
\end{equation}

By integration by parts, we have the identity
\begin{equation}
    \sum_{h \in H_{k_1+k_2}} \mathcal{L}(h)
    = -\int_{\mathbb{R}} e^{V } \dd \mathcal{B}(V)
    = \int_{\mathbb{R}} e^{V } \mathcal{B}(V) \dd V
    = e^{\mu(k_1+k_2)} \int_{\mathbb{R}} e^{V } \mathcal{B}(V+\mu(k_1+k_2)) \dd V.
\end{equation}

Under the GRH, the Littlewood-type bound (see \cite[Corollary 1.1]{Chandee2009} or \cite[\S 4]{ChandeeSoundararajan2011}) gives
\begin{equation}
    \log \mathcal{L}(h) \leq C l \frac{\log(k_1+k_2)}{\log\log(k_1+k_2)}
\end{equation}
for some constant $C > 1$. Therefore, in the integral above, we may restrict to the range
\begin{equation}
    l\sqrt{\log \log (k_1+k_2)} \leq V \leq Cl\frac{\log (k_1+k_2)}{\log\log(k_1+k_2)},
\end{equation}
while for smaller $V$ we simply use the dimension estimate for $H_{k_1+k_2}$.

Setting $x = (k_1+k_2)^{\frac{1}{\varepsilon V}}$, we observe that for
\begin{equation}
    l\sqrt{\log \log (k_1+k_2)} \leq V \leq l(\log\log (k_1+k_2))^4,
\end{equation}
we have
\begin{equation*}
    -\frac{l}{2} \log\log x + O(l\log\log\log (k_1+k_2)) \leq \mu(k_1+k_2).
\end{equation*}

From Lemma \ref{lemma:LogLfunctions} and \eqref{eqn:secondterms2}, we deduce that
\begin{equation*}
    \mathcal{B}(V+\mu(k_1+k_2)) \leq \mathcal{A}(V(1-2\varepsilon);x)
\end{equation*}
when $l\sqrt{\log\log (k_1+k_2)} \leq V \leq l(\log\log (k_1+k_2))^4$. This inequality remains valid for $V \geq l(\log\log (k_1+k_2))^4$ since in this range $V+\mu(k_1+k_2) = V(1+o(1))$.

Combining these estimates with Lemma \ref{lemma:AXY}, we obtain for some absolute constant $C > 0$:
\begin{multline}
    \sum_{h \in H_{k_1+k_2}} \mathcal{L}(h)
    \ll (k_1+k_2) e^{\mu(k_1+k_2)} \\
    \times \int_{l\sqrt{\log\log (k_1+k_2)}}^{Cl\frac{\log (k_1+k_2)}{\log\log(k_1+k_2)}} e^V \left(e^{-\frac{(1-\varepsilon)V^2}{2\sigma(k_1+k_2)^2}} (\log\log (k_1+k_2))^3 + e^{-\varepsilon V\log V}\right) \dd V \\
    \ll l(k_1+k_2) \log ^\varepsilon (k_1+k_2) e^{\mu(k_1+k_2)+\frac{\sigma(k_1+k_2)^2}{2(1-\varepsilon)}} \\
    \ll l(k_1+k_2) \log^{\frac{l(l-1)}{2}+(l^2+1)\varepsilon} (k_1+k_2),
\end{multline}
where in the final step we employed the Gaussian integral identity
\begin{equation*}
    \int_{\mathbb{R}} e^{-\frac{x^2}{2\sigma^2}+x} \dd x = \sqrt{2\pi} \sigma e^{\frac{\sigma^2}{2}}.
\end{equation*}
This completes the proof.

\end{proof}

\subsection{Proof of Theorem~\ref{thm:Heckejoint}}\label{sec:quadjoint}

\begin{proof}
When $\{f_1,f_2\} = \{f_3,f_4\} \neq \{f_1\}$, this is Theorem 1.5 in Huang~\cite{Huang2024}.
When $f_1 = f_2 \neq f_3 = f_4$, this is Theorem 1.8 in Huang~\cite{Huang2024}.

For the remaining configurations ($\{f_1,f_2\} \neq \{f_3,f_4\}$, and $f_1 \neq f_2$ or $f_3 \neq f_4$),
we proceed as in the proof of \cite[Theorem 1.8]{Huang2024},
with \cite[Proposition 5.1]{Huang2024} replaced by the two cases below.

Let $f_3\neq f_4$, we have
\begin{equation*}
  \frac{1}{k_1}
  \sum_{h \in H_{2k_1}}
  L\left(\tfrac{1}{2}, h\right)^{\frac{1}{2}} L\left(\tfrac{1}{2}, \sym^2 f_1 \times h\right)^{\frac{1}{2}}
   L\left(\tfrac{1}{2}, f_3 \times f_4 \times h\right)^{\frac{1}{2}}
  \ll_{\varepsilon} \log ^{-\frac{3}{8}+2\varepsilon}k_1,
  \end{equation*}
  and if also $f_1\neq f_2$, we have
  \begin{equation*}
 \frac{1}{k_1+k_2}
  \sum_{h \in H_{k_1+k_2}}
    L\left(\tfrac{1}{2}, f_1 \times f_2 \times h\right)^{\frac{1}{2}}
   L\left(\tfrac{1}{2}, f_3 \times f_4 \times h\right)^{\frac{1}{2}}
  \ll_{\varepsilon} \log ^{-\frac{1}{4}+2\varepsilon}(k_1+k_2).
  \end{equation*}

These two results follow directly from the proof of \cite[Proposition 5.1]{Huang2024}, as the necessary properties of the triple product $L$-functions are already provided in the proof of Proposition~\ref{prop:LboundSound}.
The only additional condition required is $f_1 \times f_2 \nsim f_3 \times f_4$, which ensures that there is no
overlap contribution in the fractional moment.

When $f_1 = f_2$, regardless of whether $f_1$ has complex multiplication, the form $f_3 \times f_4$ corresponds to an automorphic cusp form on $\GL(4)$.

When $f_1 \neq f_2$, $f_3 \neq f_4$, and $\#\{f_1,f_2,f_3,f_4\} = 3$, we may assume without loss of generality
that $f_1 = f_3$ and hence $f_2 \neq f_4$. Under the assumption $f_1 \times f_2 \nsim f_3 \times f_4$,
strong multiplicity one would imply
\[
\sum_{n \leq X} \lambda_{f_1}(n)\lambda_{f_2}(n)
\overline{\lambda_{f_3}(n)\lambda_{f_4}(n)}
=\sum_{n \leq X} \prod_{i=1}^{4} \lambda_{f_i}(n) \gg X,
\]
but this sum corresponds to the Rankin--Selberg convolution $(f_1 \times f_1) \times (f_2 \times f_4)
= (\sym^2 f_1 \boxplus 1) \times (f_2 \times f_4)$, leading to a contradiction.

When $\#\{f_1,f_2,f_3,f_4\} = 4$, we need to assume $f_1 \times f_2 \nsim f_3 \times f_4$.
\end{proof}

\section*{Acknowledgements}

The authors sincerely thank Professor Bingrong Huang for the many valuable suggestions on this paper, including the application of the case $0<p<2$, which leads to coefficient distributions dominated by small-amplitude components, as well as numerous helpful comments on the writing.
The authors would like to thank Professor Philippe Michel for his constant encouragement.
The authors were supported by NSFC
(No. 12031008) and the National Key R\&D Program of China (No. 2021YFA1000700).
The author acknowledges the support for their visit from the Centre de recherches math\'ematiques, Universit\'e de Montr\'eal.

%


\end{document}